\documentclass[12pt,reqno]{amsart}

\usepackage[margin=1in]{geometry}
\usepackage{amsmath,amssymb,amsthm,mathtools}
\usepackage{booktabs}
\usepackage{colonequals}
\usepackage{xcolor}
\usepackage[pdfusetitle,colorlinks=true,linkcolor=blue,citecolor=blue,urlcolor=blue]{hyperref}

\newtheorem{theorem}{Theorem}
\newtheorem{lemma}[theorem]{Lemma}

\theoremstyle{definition}
\newtheorem{definition}[theorem]{Definition}
\theoremstyle{remark}
\newtheorem{remark}[theorem]{Remark}
\numberwithin{equation}{section}
\numberwithin{theorem}{section}

\newcommand{\C}{\mathbb C}
\newcommand{\D}{\mathbb D}
\newcommand{\E}{\mathbb E}
\newcommand{\N}{\mathbb N}
\newcommand{\A}{\mathcal A}
\newcommand{\R}{\mathbb R}

\newcommand{\ph}{\operatorname{phase}}

\newcommand{\ip}[2]{\left\langle #1,#2\right\rangle}
\newcommand{\sgn}{\mathrm{sign}}
\renewcommand{\epsilon}{\varepsilon}

\newcommand{\KG}{K_G^{\C}}
\newcommand{\norm}[1]{\lVert#1\rVert}
\newcommand{\diag}{\operatorname{diag}}
\newcommand{\Rea}{\operatorname{Re}}
\newcommand{\Hcal}{\mathcal H}

\title{Sharper Bounds for the Complex Grothendieck Constant}
\author{Steven Heilman, Chris Jones, Giulio Malavolta}
\date{\today}

\thanks{
Email: stevenmheilman@gmail.com, chijones@ucdavis.edu, giulio.malavolta@unibocconi.it
Department of Mathematics, University of Southern California, Los Angeles, CA 90089\\
Department of Mathematics, University of California, Davis, One Shields Avenue, Davis, CA 95616\\
Department of Computing Sciences, Bocconi University, Via Sarfatti, 25, 20136 Milan, Italy}

\begin{document}

\begin{abstract}
We show that 
$
 1.4<K_G^{\mathbb{C}}
 <1.404898554746,
$
where $K_G^{\mathbb{C}}$ is the complex Grothendieck constant.  The upper bound improves on Haagerup's bound of $1.40490913\ldots$ from 1987, and the lower bound improves on Davie's bound of $1.33807$ from 1984.  The upper bound combines ideas from the real and noncommutative complex cases of Grothendieck's inequality.
\end{abstract}

\maketitle

\section{Introduction}

The complex Grothendieck constant $K_G^{\C}$ is defined \cite{groth53} as the infimum over all $K\in(0,\infty)$ such that, for any $m,n\in\N$ and for any $m\times n$ complex matrix \(a_{ij}\), we have
\begin{equation}\label{grotheq}
 \sup_{
 \substack{
 x_1,\ldots,x_m,y_1,\ldots,y_n\in\C^{m+n-1}\colon\\
 \|x_1\|=\cdots=\|x_m\|=\|y_1\|=\cdots=\|y_n\|=1
 }
 }
\left|\sum_{i=1}^{m}\sum_{j=1}^{n}a_{ij}\ip{x_i}{y_j}\right|
\leq K\cdot
\sup_{
 \substack{
 \epsilon_1,\ldots,\epsilon_m,\delta_1,\ldots,\delta_n\in\C\colon\\
 |\epsilon_1|=\cdots=|\epsilon_m|=|\delta_1|=\cdots=|\delta_n|=1
 }
 }\left|\sum_{i,j}a_{ij}\epsilon_i\overline{\delta_j}\right|.
\end{equation}
Here $\langle\cdot,\cdot\rangle$ denotes the standard complex inner product, which is linear in its first argument, and $\|\cdot\|$ denotes the standard Euclidean norm.  

Besides being a fundamental quantity in functional analysis \cite{PisierSurvey,var74,gupta18,briet12} and quantum information \cite{briet13}, the complex Grothendieck constant also has algorithmic applications.  Tropp~\cite{Tropp09} uses Grothendieck factorization (i.e. \eqref{grotheq} in an equivalent form) in a randomized polynomial-time algorithm that, given a matrix with unit-norm columns, selects a large column submatrix with bounded condition number.  For a complex matrix $A$ commuting with a transitive group of permutation matrices, \cite[Theorem~3.1]{BBLM20} gives
\[
 \|A\|_{L_\infty\to L_1}
 \leq\|A\|_{L_2\to L_2}
 \leq K_G^{\C}\|A\|_{L_\infty\to L_1},
\]
where the $L_p$ norms use normalized counting measure.  Thus the largest singular value approximates the $L_\infty\to L_1$ norm within a factor $K_G^{\C}$.

Haagerup proved \(K_G^{\C}\leq\gamma_H^{-1}\stackrel{\eqref{eq:gammaH}}{\approx}1.40490913\ldots\) in 1987 \cite{Haagerup}, which, until now, remained the smallest upper bound on $K_G^{\C}$.  Following improved upper bounds on $K_G^{\C}$ by \cite{Kaijser,Pisier78}, Haagerup used a complex analogue of Krivine's method for proving the bound $K_G^{\R}\leq\frac{\pi}{2\log(1+\sqrt{2})}$ \cite{Krivine} on Grothendieck's real constant (which has since been improved in a sequence of works \cite{BMMN,Heilman,LiEtAl,SahaEtAl}).  In particular, \cite{Haagerup} first preprocesses the vectors $x_i,y_j$, mapping them to a tensor product (Fock) space (see \eqref{one5}), where the map's coefficients are determined by the Taylor coefficients of the inverse of the function $h$ from \eqref{eq:analytic-h}, i.e. $h(t)$ is the correlation of the phase of two standard jointly complex Gaussians with correlation $t\in(-1,1)$.  Then, the resulting vectors are projected onto complex Gaussians and the phase function $z\mapsto z/|z|$ is applied (if $z\neq0$), producing $\epsilon_i,\delta_j$. The preprocessing is set up so that it cancels the nonlinearity that results from the Gaussian projection. This proof of \cite{Haagerup} is presented in Remark \ref{haark} and Lemma \ref{lem:linear} below. 

The lower bound $\KG\geq1.33807$ was shown in \cite{davie84} by computing the $\infty\to1$ norm of $P_1 - \lambda I$, where $P_1$ is the projection onto the first degree Hermite-Fourier coefficients, and $\lambda\in\R$ is chosen to optimize this norm.  The lower bound was independently improved \cite{guo26} to $\KG>1.35584631827168$ by subtracting odd Hermite projections from the identity map, as opposed to subtracting odd Hermite projections from $P_1 - \lambda I$, as was done in the real case \cite{Heilmanb,jones26}.
Our main result is the following.

\begin{theorem}\label{thm:main}
\begin{equation}
 \boxed{
 1.4<
 K_G^{\C}
 <1.404898554746
 <\gamma_H^{-1}-10^{-5}.
 }
 \label{eq:main-bound}
\end{equation}
\end{theorem}

The proof of the upper bound of Theorem \ref{thm:main} synthesizes ideas for rounding schemes of the real Grothendieck inequality \cite{BMMN,Heilman,LiEtAl,SahaEtAl} and also from the noncommutative complex case \cite{NRV}.

In particular, we use a two-stage preprocessing with a mixed rounding scheme.  First, a common preprocessor is applied to the vectors $x_i,y_j$ in \eqref{grotheq}.  Then, if the outcome of rolling a particular nonuniform six-sided die is not $6$, the vectors $x_i,y_j$ are further preprocessed in one of five different ways, four of which are small perturbations of Haagerup's preprocessing.  Then, these vectors are projected onto complex Gaussians and the phase function is applied to output $\epsilon_i,\delta_j$.  If the die outcome is $6$, then no further preprocessing occurs, the vectors are projected onto complex Gaussians, and perturbed phase functions $F,G$ from \eqref{eq:origin-FG} are applied to them (or their conjugate reflections $z\mapsto(\overline{F(\overline{z})},\overline{G(\overline{z})})$, each with $1/2$ probability).

The first five die outcomes correspond to a ``mixed rounding scheme,'' as used in \cite{BMMN,NaorRegev,LiEtAl,SahaEtAl}.  The sixth outcome is a relative of the rounding scheme from \cite{NRV}.  The novelty in this construction is both the two-stage process of the rounding scheme, and the combination of different preprocessing functions.  For comparison, the mixed rounding schemes from \cite{BMMN,NaorRegev,LiEtAl,SahaEtAl} use a single pair of preprocessing maps.

If instead of the upper bound of Theorem \ref{thm:main}, the reader is satisfied by the weaker statement $K_G^{\C}<\gamma_H^{-1}$, then a simpler proof than that used for Theorem \ref{thm:main} suffices.  That simpler rounding scheme is analogous to the one described above, where the first five alternatives collapse to one case, so there are only two alternatives instead of six.  Since the statement $K_G^{\C}<\gamma_H^{-1}$ is weaker than Theorem \ref{thm:main}, we omit the details.  Due to the different preprocessing steps in this two-stage mixed rounding scheme, the statement $K_G^{\C}<\gamma_H^{-1}$ does not follow from the most straightforward adaptation of the argument of \cite{BMMN}.

The lower bound of Theorem \ref{thm:main} follows by upper bounding the $\infty\to1$ norm of
\[
 T_n\colonequals P_{1,0}-\frac{37}{50}P_{2,1},
\]
where $P_{p,q}$ is the orthogonal projection onto the complex Hermite component of bidegree $(p,q)$ under standard Gaussian measure on $\C^n$.  We prove in Theorem \ref{thm:main2} that
\[
 \sup_{n\geq1}\|T_n\|_{\infty\to1}
 <\frac{2857}{4000}<\frac57,
\]
which implies $K_G^{\C}>4000/2857>1.4$. 

Surprisingly, we were unable to successfully apply the ``limiting Krivine schemes'' directly from \cite{SahaEtAl} in Theorem \ref{thm:main}, although those schemes achieve the best bounds on the real constant $K_G^{\R}$.  There it is shown that certain nonlinear maps can be applied to the correlation parameter.  Perhaps these ideas could be incorporated somehow to improve \eqref{eq:main-bound}. 

Both the upper and lower bounds in Theorem \ref{thm:main} perform analytical reductions to inequalities which are certified numerically.
Supporting codes for Theorem \ref{thm:main} are given at:

\url{https://github.com/sheilman77/grothendieck_cx}

\subsection{A Remaining Conjecture}
Haagerup \cite{Haagerup} conjectured that a full complex analogue of Krivine's
argument should imply that $K_G^\C$ is equal to
\begin{equation}
 \gamma_\dagger^{-1}
 =1.404575934663742\ldots,
 \qquad
 \gamma_\dagger
 \colonequals
 \int_0^{\pi/2}\frac{\cos^2t}{\sqrt{1+\sin^2t}}\,dt,
 \label{eq:haagerup-conjecture}
\end{equation}
(see \cite[Section~4]{PisierSurvey}), but this remains a conjecture.

\section{Preliminaries on Haagerup's Approach}

A standard complex Gaussian has density
\(w\mapsto\pi^{-1}e^{-|w|^2}\), $w\in\C$, so its second absolute moment is one.
Let \((X,Y)\) be a jointly proper standard complex Gaussian pair with
\(\E[X\overline Y]=z\), \(|z|\leq1\).  Haagerup's phase identity \cite[Lemma~3.2]{Haagerup} (see also \cite[Lemma~2.1]{FLZ}) is
\begin{equation}
 \E\bigl[\ph(X)\overline{\ph(Y)}\bigr]
 =\mathfrak h(z)\colonequals\frac{\pi}{4}z\cdot\,
 {}_2F_1\!\left(\frac12,\frac12;2;|z|^2\right)
 =\sum_{n\geq0}h_nz|z|^{2n},
 \label{eq:haagerup-phase}
\end{equation}
where \(\ph(w)=w/|w|\) for all $w\in\C\setminus\{0\}$ with $\ph(0)\colonequals1$, and
\begin{equation}
 h_k=\frac{\pi}{4(k+1)}\frac{\binom{2k}{k}^2}{16^k},
 \qquad
 \frac{h_{k+1}}{h_k}
 =\frac{(2k+1)^2}{4(k+1)(k+2)},
 \qquad \sum_{k\geq0}h_k=1.
 \label{eq:h-coefficients}
\end{equation}

Write \(\D\colonequals\{z\in\C\colon|z|<1\}\).  For any real $t\in(-1,1)$, let
\begin{equation}
 h(t)\colonequals\frac{\pi}{4}t\cdot\,
 {}_2F_1\!\left(\frac12,\frac12;2;t^2\right)
 =\sum_{k\geq0}h_kt^{2k+1}.
 \label{eq:analytic-h}
\end{equation}
Thus \(h\) is an analytic odd series, whereas \(\mathfrak h\) is the
physical charge-one radial kernel.  Here ``charge one'' means
\(K(e^{i\theta}z)=e^{i\theta}K(z)\) for all $\theta\in\R$ and $z\in\overline\D$.  The integral form in~\cite[Lemma~3.2]{Haagerup}
also yields
\begin{equation}\label{hint}
 \mathfrak h(z)=z\int_0^{\pi/2}
 \frac{\cos^2\theta}{\sqrt{1-|z|^2\sin^2\theta}}\,d\theta,\qquad\forall\,z\in\overline\D.
\end{equation}

Let \(x_H\in(0,1)\) be the unique solution of
\begin{equation}
 2x_H\cdot\,{}_2F_1\!\left(\frac12,\frac12;2;x_H^2\right)=1+x_H,
 \label{eq:contact}
\end{equation}
and put
\begin{equation}
 \gamma_H\colonequals\frac{\pi(1+x_H)}8
 \approx0.711789806684998.
 \label{eq:gammaH}
\end{equation}
(Recall $h(0)=0$ by \eqref{eq:analytic-h} and $h(1)=1$ from \eqref{eq:h-coefficients}, so the function $\phi(x)\colonequals h(x)-(\pi/8)(1+x)$ is negative at $x=0$ and positive at $x=1$, so there exists at least one $x$ satisfying \eqref{eq:contact}.  To see uniqueness, recall that $h_0=\pi/4$ and $h_k>0$ for all $k\geq1$, so $\phi'(x)>0$ for all $0<x<1$.)

Let \(h^{-1}\) be the odd analytic inverse germ of
\eqref{eq:analytic-h}, and let
\begin{equation}\label{rhostdef}
\rho_*(t)\colonequals h^{-1}(\gamma_Ht).
\end{equation}

\section{General Results on Kernels}

We first isolate general facts in Haagerup's rounding procedure in terms of kernels.  Let \(S(E)\) denote the unit sphere
of a complex Hilbert space \(E\).

\begin{definition}\label{def:allowable}
A map \(\rho\colon\overline\D\to\overline\D\) is an \textbf{allowable
preprocessor} if, for every complex Hilbert space \(E\), there
are a complex Hilbert space \(F\) and maps
\(U,V\colon S(E)\to S(F)\) where
\[
 \ip{U(x)}{V(y)}=\rho\bigl(\ip{x}{y}\bigr)
 \qquad\forall\,x,y\in S(E).
\]
A function \(K:\overline\D\to\overline\D\) is a \textbf{valid rounding
kernel} if, $\forall$ $E$, $\forall$ $m,n\in\N$ and $\forall$ $x_1,\ldots,x_m,$  $y_1,\ldots,y_n\in S(E)$, there are random variables $\epsilon_1,\ldots,\epsilon_m,\delta_1,\ldots,\delta_n\in S(\C)$ such that
\begin{equation}\label{valideq}
 \E[\varepsilon_i\overline{\delta_j}]
 =K\bigl(\ip{x_i}{y_j}\bigr),\qquad\forall\,1\leq i\leq m,\, 1\leq j\leq n.
\end{equation}
\end{definition}

\begin{lemma}[Kernel Bound]
\label{lem:linear}
Let $r>0$.  Assume $K(z)=rz$ is valid.  Then $K_G^\C\leq r^{-1}$.
\end{lemma}
\begin{proof}
Choose unit vectors $x_1,\ldots,x_m,y_1,\ldots,y_n$ attaining the supremum on the left side of \eqref{grotheq}.
Since \(K(z)=rz\) is a valid kernel,
\begin{flalign*}
\left|\sum_{i=1}^{m}\sum_{j=1}^{n}a_{ij}\langle x_i, y_j\rangle\right|
&=\frac{1}{r}\left|\sum_{i=1}^{m}\sum_{j=1}^{n}a_{ij}K(\langle x_i, y_j\rangle)\right|\\
 &\stackrel{\eqref{valideq}}{=}\frac{1}{r}\left|\E\sum_{i=1}^{m}\sum_{j=1}^{n}a_{ij}
\varepsilon_i\overline{\delta_j}\right|
\leq\frac{1}{r}\sup_{\epsilon_1,\ldots,\epsilon_m,\delta_1,\ldots,\delta_n\in S(\C)}\left|\sum_{i=1}^{m}\sum_{j=1}^{n}a_{ij}
       \varepsilon_i\overline{\delta_j}\right|.
\end{flalign*}

\end{proof}

We use the real odd Wiener space
\[
 \A=\left\{u(t)=\sum_{k\geq0}c_kt^{2k+1}:
      c_k\in\mathbb R,\ \|u\|_{\A}\colonequals\sum_k|c_k|<\infty\right\}.
\]
The radial extension of $u\in\A$ is
\begin{equation}\label{physdef}
 u^{\mathrm{phys}}(z)\colonequals\sum_{k\geq0}c_kz|z|^{2k}.
\end{equation}
Coefficient convolution and the triangle inequality give
\begin{equation}
 \|u^{2j+1}\|_{\A}
 \leq\left(\sum_{k\geq0}|c_k|\right)^{2j+1}
 =\|u\|_{\A}^{2j+1},
\label{eq:wiener-submultiplicative}
\end{equation}
for all $j\geq0$.
In particular, if \(\|u\|_{\A}\leq1\), then
\begin{equation}
 \sum_{k\geq0}h_k\|u^{2k+1}\|_{\A}
 \leq\sum_{k\geq0}h_k\|u\|_{\A}^{2k+1}
 \leq\sum_{k\geq0}h_k=1.
 \label{eq:absolute-composition}
\end{equation}
Thus \(h\circ u\) converges absolutely in \(\A\).  Expansion shows
that its radial extension is
\(\mathfrak h\circ u^{\mathrm{phys}}\).

\begin{remark}\label{haark}
Recall $\gamma_H\approx0.71178980668\ldots$ from \eqref{eq:gammaH}.  Let $\rho_*(t)\colonequals h^{-1}(\gamma_H t)$ for all $t\in[0,1]$.  Then \cite{Haagerup} shows that $\|\rho_*\|_\A=1$, so $K_{\rho_*}(z)=\gamma_H\cdot z$ is valid by Lemma \ref{lem:wiener} below, so Lemma \ref{lem:linear} implies that $K_G^\C\leq \gamma_H^{-1}$, recovering the main result of \cite{Haagerup}.  To justify $\|\rho_*\|_\A=1$, write
$
\rho_*(t)=\sum_{n\geq0}c_nt^{2n+1}.
$
The known expansion of $h^{-1}$ shows \(c_0=(1+x_H)/2\), $c_1<0$, \(c_2=0\), and $c_n<0$ for every $n>2$.  Since
\(\rho_*(1)=h^{-1}(\gamma_H)=x_H\) by \eqref{eq:contact},
\begin{equation}\label{two5}
       \|\rho_*\|_{\A}
       =c_0-\sum_{n\geq1}c_n
       =2c_0-\sum_{n\geq0}c_n
       =2c_0 -x_H
       \stackrel{\eqref{eq:contact}}{=}1,
 \qquad h(\rho_*(t))=\gamma_H t.
\end{equation}
\end{remark}
\begin{remark}
The sign pattern of the coefficients of $h^{-1}$ here is different from the analogous sign pattern for real scalars.  In the latter setting, the inverse of the Gaussian correlation of two sign functions is a sine function, whose coefficients have alternating signs.  But $c_n<0$ for all $n>2$, so bounds for $K_G^{\R}$ do not carry over directly to the complex case.
\end{remark}

\begin{lemma}[Existence and Convexity of Valid Kernels]
\label{lem:wiener}
If \(\rho(t)=\sum_{k\geq0}c_kt^{2k+1}\in\A\) and
\(\|\rho\|_{\A}\leq1\), then \(\rho^{\rm phys}\) is an allowable preprocessor.
For any $z\in\overline\D$ define
\begin{equation}\label{krhodef}
       K_\rho(z)\colonequals\mathfrak h(\rho^{\mathrm{phys}}(z)).
\end{equation}
Then $K_\rho$ is a valid rounding correlation kernel, with coefficient representative \(h\circ\rho\in\A\).

More generally, let \(F,G\colon \C\to S(\C)\) be measurable.  For a jointly
proper standard circular complex Gaussian pair \((X_z,Y_z)\) satisfying
\(\E[X_z\overline{Y_z}]=z\) for all $z\in\overline\D$, the function
\begin{equation}\label{kfgdef}
       K_{F,G}(z)\colonequals\E\big[F(X_z)\overline{G(Y_z)}\big],\qquad\forall\,z\in\overline\D,
\end{equation}
is a valid complex correlation rounding kernel.  

Convex combinations of valid rounding kernels are valid rounding kernels.
\end{lemma}

\begin{proof}
Put \(s_\rho\colonequals\sum_{k\geq0}|c_k|\).  Note that $s_\rho\leq1$ by the assumption $\|\rho\|_{\A}\leq1$.  After placing the tensor summands in mutually
orthogonal Hilbert spaces, fix unit vectors \(e_L,e_R\in H\) that are orthogonal
to every tensor summand and to each other, and set
\begin{equation}\label{one5}
\begin{aligned}
 U(x)&\colonequals
 \bigoplus_{k\geq0}|c_k|^{1/2}(x^{\otimes(k+1)}\otimes \overline x^{\otimes k})
              \oplus\Big((1-s_\rho)^{1/2}\,e_L\Big),
 \\
 V(y)&\colonequals\bigoplus_{k\geq0}\sgn(c_k)|c_k|^{1/2}(y^{\otimes(k+1)}\otimes \overline y^{\otimes k})
              \oplus\Big((1-s_\rho)^{1/2}\,e_R\Big).
\end{aligned}
\end{equation}
Then $\|U(x)\|=\|V(y)\|=1$ by definition of $s_\rho$, and
\begin{equation}\label{one6}
       \ip{U(x)}{V(y)}
       \stackrel{\eqref{one5}}{=}\sum_{k\geq0}c_k\ip{x}{y}^{k+1}
                         \overline{\ip{x}{y}}^{\,k}
       \stackrel{\eqref{physdef}}{=}\rho^{\mathrm{phys}}(\ip{x}{y}).
\end{equation}
Thus the maps $U,V$ are feature maps witnessing that \(\rho^{\rm phys}\) is an
allowable preprocessor.

Let $x_1,\ldots,x_m,y_1,\ldots,y_n\in S(H)$.  Let $E\colonequals\mathrm{span}_{\C}\{U(x_1),\ldots,U(x_m),V(y_1),\ldots,V(y_n)\}$.  Then $E$ is a linear subspace with dimension at most $m+n$.  Let $g$ be a complex standard Gaussian in $E$.  For every $v,w\in E$, our convention that the inner product is linear
in its first argument gives
\begin{equation}\label{evf}
 \mathbb E\!\left[
   \langle v,g\rangle\,
   \overline{\langle w,g\rangle}
 \right]
 =\langle v,w\rangle,
 \qquad
 \mathbb E\!\left[
   \langle v,g\rangle\langle w,g\rangle
 \right]=0.
\end{equation}
Consequently,
\[
 Z_i\colonequals\langle U(x_i),g\rangle,
 \qquad
 W_j\colonequals\langle V(y_j),g\rangle,\qquad\forall\,1\leq i\leq m,\, 1\leq j\leq n,
\]
form a jointly proper family of standard circular complex Gaussians,
and \eqref{one6} yields
\begin{equation}\label{ppeq}
 \mathbb E[Z_i\overline{W_j}]
 \stackrel{\eqref{evf}}{=}\langle U(x_i),V(y_j)\rangle
 \stackrel{\eqref{one6}}{=}\rho^{\mathrm{phys}}(\langle x_i,y_j\rangle),\qquad\forall\,1\leq i\leq m,\,1\leq j\leq n.
\end{equation}
Thus, on setting
$\epsilon_i\colonequals\operatorname{phase}(Z_i)$ and
$\delta_j\colonequals\operatorname{phase}(W_j)$, identity \eqref{eq:haagerup-phase} gives
\[
 \mathbb E[\epsilon_i\overline{\delta_j}]
 \stackrel{\eqref{eq:haagerup-phase}\wedge\eqref{ppeq}}{=}\mathfrak{h}\!\left(
   \rho^{\mathrm{phys}}(\langle x_i,y_j\rangle)
 \right)
 \stackrel{\eqref{krhodef}}{=}K_\rho(\langle x_i,y_j\rangle).
\]
This proves the required validity for every finite instance.

It remains to verify the assertion for 
$K_{F,G}$.  Let
$
 x_1,\ldots,x_m,y_1,\ldots,y_n\in S(H),
$
and let
$
 E'\colonequals\operatorname{span}_{\mathbb C}
 \{x_1,\ldots,x_m,y_1,\ldots,y_n\}.
$
Then $\dim E'\leq m+n$.  Choose a standard proper circular complex
Gaussian vector $g'$ on $E'$, and define the shared Gaussian projections
\[
 X_i\colonequals\langle x_i,g'\rangle,
 \qquad
 Y_j\colonequals\langle y_j,g'\rangle,\qquad\forall\,1\leq i\leq m,\, 1\leq j\leq n.
\]
The entire family $(X_1,\ldots,X_m,Y_1,\ldots,Y_n)$ is jointly proper
complex Gaussian.  Moreover,
\[
 \mathbb E|X_i|^2=\mathbb E|Y_j|^2=1,
 \qquad
 \mathbb E[X_i\overline{Y_j}]
 =\langle x_i,y_j\rangle,\qquad\forall\,1\leq i\leq m,\, 1\leq j\leq n.
\]
Thus, for each pair $(i,j)$, the joint distribution of $(X_i,Y_j)$ is
the distribution of a jointly proper standard circular complex
Gaussian pair $(X_z,Y_z)$ with
$
 z=\langle x_i,y_j\rangle.
$
Here we use that a centered proper complex Gaussian pair is
determined by its covariance matrix.  Now set
\begin{equation}\label{epsdef}
 \epsilon_i\colonequals F(X_i),
 \qquad
 \delta_j\colonequals G(Y_j).
\end{equation}
Since $F$ and $G$ take values in $S(\C)$, we have $|\epsilon_i|=|\delta_j|=1$ for all $1\leq i\leq m$, $1\leq j\leq n$, and
\[
 \mathbb E[\epsilon_i\overline{\delta_j}]
 \stackrel{\eqref{epsdef}}{=}\mathbb E\!\left[F(X_i)\overline{G(Y_j)}\right]
 \stackrel{\eqref{kfgdef}}{=}K_{F,G}(\langle x_i,y_j\rangle).
\]
The use of one shared Gaussian vector $g'$ is important: it places all
the variables $\epsilon_1,\ldots,\epsilon_m,$ $\delta_1,\ldots,\delta_n$
on a single probability space.  The definition of a valid kernel
requires such a joint realization, rather than a separate Gaussian
pair for each $(i,j)$.  Hence $K_{F,G}$ is a valid complex rounding
correlation kernel.

Finally, validity is preserved under convex combinations.  Indeed,
let $K^{(1)},\ldots,K^{(r)}$ be valid kernels and let
$\lambda_1,\ldots,\lambda_r\geq0$ satisfy
$\sum_{a=1}^r\lambda_a=1$.  For the fixed input $x_1,\ldots,x_m,y_1,\ldots,y_n\in S(H)$, choose for
each $1\leq a\leq r$ a joint unimodular family
\[
 \epsilon_1^{(a)},\ldots,\epsilon_m^{(a)},
 \delta_1^{(a)},\ldots,\delta_n^{(a)}\in S(\C),
\]
witnessing the validity of $K^{(a)}$.  Let $A\in\{1,\ldots,r\}$ be an independent
random selector with
$
 \mathbb P(A=a)=\lambda_a,
$
for all $1\leq a\leq r$.  Define
\[
 \epsilon_i\colonequals\epsilon_i^{(A)},
 \qquad
 \delta_j\colonequals\delta_j^{(A)},
 \qquad\forall\,1\leq i\leq m,\,1\leq j\leq n.
\]
Conditioning on $A$ gives
\[
 \mathbb E[\epsilon_i\overline{\delta_j}]
 =\mathbb E(\mathbb E[\epsilon_i\overline{\delta_j}|A])
=\sum_{a=1}^r\lambda_a\,
   \mathbb E\!\left[
\epsilon_i^{(a)}\overline{\delta_j^{(a)}}
   \right]  \\
 =\sum_{a=1}^r\lambda_a
   K^{(a)}(\langle x_i,y_j\rangle),\quad\forall\,1\leq i\leq m,\,1\leq j\leq n.
\]
Therefore $\sum_{a=1}^r\lambda_aK^{(a)}$ is a valid rounding kernel. 

\end{proof}

\begin{lemma}[Repairing a nonlinear kernel]
\label{lem:inverse-repair}
Let $r\in\A$ with no degree one term.  Let
\[
 q(t)\colonequals at+r(t)\in\A,
 \qquad a>0,
 \qquad \|r\|_{\A}\leq\delta<a.
\]
For all \(0\leq\Gamma\leq a-\delta\), the inverse germ
\(\eta(t)=q^{-1}(\Gamma t)\) belongs to \(\A\) and
\(\|\eta\|_{\A}\leq1\).  Moreover, if \(q\) represents a valid
charge-one kernel, then inverse preprocessing by
\(\eta^{\mathrm{phys}}\) produces the linear valid kernel
\(z\mapsto\Gamma z\).
\end{lemma}

\begin{proof}
Write \(r(t)=\sum_{m\geq1}r_mt^{2m+1}\) and
\(r^{\#}(t)\colonequals\sum_{m\geq1}|r_m|t^{2m+1}\).  Starting with \(v_0=0\),
define formal power series $v_{k+1}$ with nonnegative coefficients by
\[
 v_{k+1}
 \colonequals\frac{\Gamma t+r^{\#}(v_k)}a.
\]
Since $r^\#$ has nonnegative coefficients, $v_k\leq v_{k+1}$ coefficientwise.  If \(\|v_k\|_{\A}\leq1\),
then submultiplicativity of the $\|\cdot\|_\A$ norm gives
\[
 \|v_{k+1}\|_{\A}
 \leq\frac{\Gamma+\delta}{a}\leq1,\qquad\forall\,k\geq0.
\]
Coefficientwise monotone convergence therefore gives a $k\to\infty$ limit
\(v\in\A\) with \(\|v\|_{\A}\leq1\): monotone convergence gives
\(\sum_{k\geq1}[t^k]v(t)\leq1\), and coefficientwise passage in the recursion gives
\(v=(\Gamma t+r^{\#}(v))/a\).  The formal equation
\begin{equation}\label{formaleq}
 \eta=\frac{\Gamma t-r(\eta)}a
\end{equation}
has a unique formal power series solution with zero constant term, since \(r\) has order at least three.  Induction
on degree gives
\(\lvert[t^k]\eta\rvert\leq[t^k]v\) for every \(k\), so
\(\eta\in\A\) and \(\|\eta\|_{\A}\leq1\).  Absolute convergence upgrades
the formal identity \eqref{formaleq} to \(q(\eta(t))=\Gamma t\) in \(\A\).

By Lemma~\ref{lem:wiener}, \(\eta^{\mathrm{phys}}\) is allowable.  Radial
extension respects composition: writing
\(\eta(t)=tA(t^2)\) and \(q(t)=tB(t^2)\), one has
\[
 q^{\mathrm{phys}}\bigl(\eta^{\mathrm{phys}}(z)\bigr)
 =zA(|z|^2)B\bigl(|z|^2A(|z|^2)^2\bigr)
 =(q\circ \eta)^{\mathrm{phys}}(z)=\Gamma z.
\]
Apply the valid kernel represented by \(q\) to the two families of
feature vectors realizing \(\eta^{\mathrm{phys}}\).  This proves validity
of the composed linear kernel.
\end{proof}

\section{Choice of Kernels}

\subsection{Preprocessing the phase kernel}

We now describe the first five branches of our mixed rounding scheme, which use phase rounding after different preprocessings.  Let

$$
\rho(t)=\sum_{m\geq0}c_mt^{2m+1}=t\,p(t^2),
\qquad
p(x)=\sum_{m\geq0}c_mx^m,
\qquad
c_m\in\mathbb R,\quad \sum_{m\geq0}|c_m|\leq1.
$$

By the existence lemma for valid kernels (Lemma \ref{lem:wiener}), the radial extension
$\rho^{\mathrm{phys}}(z)=z\,p(|z|^2)$
is an allowable preprocessor. Composing this preprocessor with phase rounding produces the valid kernel

$$
K_\rho(z)=\mathfrak h\bigl(\rho^{\mathrm{phys}}(z)\bigr).
$$
Its odd analytic representative is \(h\circ\rho\).
For example, Haagerup’s preprocessor
$\rho_*(t)=h^{-1}(\gamma_Ht)$
belongs to this family and gives the exactly linear kernel
\(K_{\rho_*}(z)=\gamma_Hz\).
To compute the coefficients of \(K_\rho\), put \(x=|z|^2\) and
\(\alpha_j=h_j/\pi\). Since \(p\) has real coefficients, the expansion of \(\mathfrak h\) gives

\begin{equation}\label{eq:scalar-profile}
K_\rho(z)
=\pi z\sum_{j\geq0}\alpha_jx^jp(x)^{2j+1},
\qquad
\alpha_0=\frac14,\qquad
\frac{\alpha_{j+1}}{\alpha_j}
=\frac{(2j+1)^2}{4(j+1)(j+2)}.
\end{equation}

In particular, when the coefficients of \(\rho\) are rational, the coefficients of \(K_\rho\) can be computed exactly over \(\mathbb Q\), apart from the common factor \(\pi\).




\subsection{The sixth kernel}\label{subsec:origin}

We now define the sixth alternative.  Let
\(P(s)=\sum_{j=0}^9P_js^j\), where the following displayed decimals are
exact terminating rationals:
\begingroup\small
\begin{flalign*}
 &(\Re P_0,\ldots,\Re P_9)
 =
 (
-0.08698387653870097,
-1.2330310998031548,
0.7731241827583651,\\
&\qquad-0.22194697498795207,
0.02812176806963062,
-0.000900511816912206,\\
&\qquad-0.00016426639538512312,
0.00001934672964561798,
-0.0000007806468542918022,\\
&\qquad0.000000010807705028012223),\\
 &(\Im P_0,\ldots,\Im P_9)
 =(-0.9438672688639868,
1.7787014426224697,
-1.2883230878434588,\\
&\qquad0.37871541533769754,
-0.046163071426838245,
0.00035393667390378955,\\
&\qquad0.0004808219961183191,
-0.00004805019315921306,
0.0000018538964354898714,\\
&\qquad-0.0000000252261077693192).
\end{flalign*}
\endgroup
A simple univariate polynomial certificate proves
\begin{equation}
 |P(s)|>\frac35\qquad\forall\,s\geq0.
 \label{eq:P-nonzero}
\end{equation}
Put \(u(s)\colonequals P(s)/|P(s)|\).  For any $t\in\R$, write
$
 \operatorname{cay}(t)=\frac{1-t^2+2it}{1+t^2}\in S(\C).
$
Define the following exact parameters:
\begin{align}
C&\colonequals\operatorname{cay}\!\left(\frac{356080252073}{10^{12}}\right),
 &
A&\colonequals\operatorname{cay}\!\left(-\frac{1775}{10000}\right),
 &
 \delta&\colonequals\frac{11}{500},
 \label{eq:origin-parameters}\\
 u_0&\colonequals u(0),&
 z_0&\colonequals\overline C\,u_0^2,
 &
 D&\colonequals-iz_0,\qquad B\colonequals DA.
 \label{eq:origin-parameters-2}
\end{align}
Although \(u_0\) need not have rational coordinates, \(u_0^2\) does;
hence so do \(z_0,D,B\).  Define
\begin{equation}
 f(s)\colonequals u(s)\ph(s+\delta A),
 \qquad
 g(s)\colonequals C\overline{u(s)}\ph(s+\delta B),
 \label{eq:origin-fg}
\end{equation}
\begin{equation}
 F(z)\colonequals\ph(z)f(|z|^2),\qquad
 G(z)\colonequals\ph(z)g(|z|^2).
 \label{eq:origin-FG}
\end{equation}
Note that $F,G\colon\C\to S(\C)$.  Equations \eqref{eq:P-nonzero} and
\(\Im A,\Im B\neq0\) show that these phases are well-defined away from
irrelevant Gaussian null sets.  Moreover,
\begin{equation}
 f(0)\overline{g(0)}=z_0A\overline B=i.
 \label{eq:endpoint-cancellation}
\end{equation}
This exact cancellation removes the slow \(h_n\) term from the
real-coefficient tail obtained by the conjugation symmetrization below.

By Lemma~\ref{lem:wiener}, their direct phase-rounding kernel is valid.  It is charge-one equivariant, since $K(z)=\E[F(X_z)\overline{G(Y_z)}]$, so for any $\omega\in S(\C)$, since $(\omega X_z, Y_z)$ has covariance $\omega z$, we have $K_{F,G}(\omega z)=\E[F(\omega X_z)\overline{G(Y_z)}]=\omega K_{F,G}(z)$.  

\begin{lemma}
There exist complex numbers $(c_k)_{k\geq0}$ with $\sum_{k\geq0}|c_k|\leq1$ such that
\begin{equation}\label{three2}
K(z)=K_{F,G}(z)=\sum_{k\geq0}c_kz^{k+1}\overline z^{\,k},\qquad\forall\,z\in\overline\D.
\end{equation}
\end{lemma} 
\begin{proof}
Let \(L_k^{(1)}\) be the generalized Laguerre polynomial and define
\begin{equation}\label{phidef}
L_k^{(1)}(x)\colonequals \frac{x^{-1}e^x}{k!}\frac{d^k}{dx^k}(e^{-x} x^{k+1})
,\qquad
       \phi_k(z)\colonequals \frac{zL_k^{(1)}(|z|^2)}{\sqrt{k+1}}.
\end{equation}
\begin{equation}
\begin{aligned}
 a_k&\colonequals\int_0^\infty e^{-s}
 \frac{\sqrt{s}L_k^{(1)}(s)}{\sqrt{k+1}}f(s)\,ds
 =\langle F,\phi_k\rangle,\\
b_k&\colonequals\int_0^\infty e^{-s}
 \frac{\sqrt{s}L_k^{(1)}(s)}{\sqrt{k+1}}g(s)\,ds
 =\langle G,\phi_k\rangle.
 \end{aligned}
 \label{eq:origin-coefficients}
\end{equation}
Then $(\phi_k)_{k\geq0}$ are normalized
charge-one complex Hermite functions, since
\[
 \int_0^\infty se^{-s}L_j^{(1)}(s)L_k^{(1)}(s)\,ds
 =(k+1)\mathbf1_{\{k=j\}}.
\]
In particular, $\E|\phi_k(Z)|^2=1$.  By completeness of the generalized Laguerre polynomials,
\((\phi_k)_{k\geq0}\) is an orthonormal basis of the charge-one subspace
of complex Gaussian \(L_2\).  

Also, $\phi_k(\omega z)=\omega\phi_k(z)$ for all $\omega\in S(\C)$.  By Laguerre orthogonality, $\phi_k$ is orthogonal to every charge-one polynomial of smaller degree.  Also $\phi_k$ has bidegree $(k+1,k)$, so the span of $\phi_k$ is the charge-one Hermite component $\mathcal H_{k+1,k}$.  Let \(\Pi_k\) denote
the orthogonal projection onto the span of $\phi_k$.  Here all function inner
products and \(L_2\) norms are taken with respect to standard complex
Gaussian measure.  

Let $z\in\overline\D$.  The Mehler identity says if $H\in\mathcal H_{p,q}$ is a Hermite component with bidegree $p,q$, then $T_z H = z^{p}\overline{z}^q H$, where
$$T_z\psi(y)
\colonequals\E \psi(zy+W\sqrt{1-|z|^2}),\qquad\forall\,y\in\C,\,\,\forall\,\psi\colon\C\to\C\text{ with }\E|\psi(Y)|^2<\infty,$$
where $Y,W$ are i.i.d. standard complex Gaussians, $X_z\colonequals zY+W\sqrt{1-|z|^2}$ and $\E X_z\overline Y=z$,
\begin{equation}\label{kbid}
K(z)
\stackrel{\eqref{kfgdef}}{=}\E[F(X_z)\overline{G(Y)}]=\langle T_z F,G\rangle.
\end{equation}
Since $F,G$ are charge-one equivariant almost surely ($F(\omega z)=\omega F(z)$ for all $\omega\in S(\C)$ and a.e. $z\in\C$), their Hermite decompositions have only $\Pi_k$ components, i.e. $F=\sum_{k\geq 0}\Pi_k F$, hence $T_z F=\sum_{k\geq 0}z^{k+1}\overline{z}^k \Pi_k F$, so that 
\begin{equation}\label{kzdef}
K(z)
\stackrel{\eqref{kbid}}{=}\langle T_z F,G\rangle
=\sum_{k\geq0}z^{k+1}\overline{z}^k\langle\Pi_k F,G\rangle
=\sum_{k\geq0}z^{k+1}\overline{z}^k\langle\Pi_k F,\Pi_k G\rangle.
\end{equation}
Define now for any $k\geq0$,
\begin{equation}\label{bkid}
 c_k\colonequals\ip{\Pi_kF}{\Pi_kG}
 \stackrel{\eqref{phidef}}{=}\ip{F}{\phi_k}\ip{\phi_k}{G}
 \stackrel{\eqref{eq:origin-coefficients}}{=}a_k\overline{b_k}.
\end{equation}

By \eqref{bkid}, 
Hermite orthogonality and the Cauchy--Schwarz inequality give
\begin{equation}\label{three3}
 \sum_{k\geq0}|c_k|
 \leq\left(\sum_{k\geq0}\|\Pi_kF\|_2^2\right)^{1/2}
      \left(\sum_{j\geq0}\|\Pi_jG\|_2^2\right)^{1/2}\leq1.
\end{equation}

\end{proof}

Let
\[
 F^\#(z)\colonequals\overline{F(\overline z)},\qquad
 G^\#(z)\colonequals\overline{G(\overline z)},\qquad\forall\,z\in\C.
\]
In contrast to \eqref{three2}, $K_{F^\#,G^\#}(z)=\sum_{k\geq0}\overline{c_k}z^{k+1}\overline z^{\,k}$ by \eqref{bkid} since $\phi_k(\overline{z})=\overline{\phi_k(z)}$, $\forall$ $k\geq0$.

By Lemma \ref{lem:wiener}, \begin{equation}\label{eq:origin-kernel}
K_{\mathrm{org}}\colonequals(K_{F,G}+K_{F^\# ,G^\#})/2
\end{equation}
is a valid kernel.  Moreover, its $n^{th}$ coefficient is
\begin{equation}\label{kndef}
k_n\colonequals (c_n + \overline{c_n})/2=\Re(c_n)\in\R.
\end{equation}
Finally, the triangle inequality implies that $\sum_{n\geq0}|k_n|\leq 1$.

\section{The mixed rounding scheme}\label{sec:candidate}

In this section and the two certification sections, \(\rho_r\),
\(1\leq r\leq5\), denotes an odd analytic representative in \(\A\), and
\(\rho_r^{\mathrm{phys}}\) denotes its radial extension.  For any
\(\rho\in\A\) whose radial extension is allowable, write
\begin{equation}
 K_\rho(z)\colonequals\mathfrak h\bigl(\rho^{\mathrm{phys}}(z)\bigr).
 \label{eq:scalar-kernel-notation}
\end{equation}
The five scalar preprocessors used in the final candidate are now defined
directly.  Every decimal in \eqref{eq:low-scalar-preprocessors} is an exact terminating
rational, not a rounded display value:
\begin{align}
 \rho_1(t)={}& 0.906231648312t-0.093651693922t^3
 -0.000116657762t^9,\notag\\
 \rho_2(t)={}& 0.910377315558t-0.063651630829t^3
 -0.025971053610t^7,\notag\\
 \rho_3(t)={}&0.912512978184t-0.035570362029t^3
 -0.050646289677t^7-0.001270370107t^9,\notag\\
 \rho_4(t)={}& 0.901913867354t-0.037111896010t^3
 -0.026391249065t^7\notag\\
 &-0.033420629518t^{11}-0.001162358049t^{15}.
 \label{eq:low-scalar-preprocessors}
\end{align}
Note that these are different perturbations of $\rho_*$ from \eqref{rhostdef}, since by \cite{Haagerup},
$$\rho_*(t)
\approx 0.9062789294t
-0.093045561673t^3
-0.00049038053603t^7
-0.00012586558875t^9
-\cdots
$$
For the fifth (high-degree) preprocessor, put
\begin{equation}
 \rho_5(t)=\sum_m c_{5,m}t^{2m+1},
 \label{eq:rho-5}
\end{equation}
where every coefficient not listed here is zero and the listed values are
again exact:
\begin{equation}
\begin{array}{c@{\;}r@{\qquad}c@{\;}r@{\qquad}c@{\;}r}
m&c_{5,m}&m&c_{5,m}&m&c_{5,m}\\ \hline
0& 0.094626045381&1&-0.856254413568&10& 0.001274576105\\
13&0.008312996734&14&0.001885453067&15&0.001662844635\\
16&0.009295485863&17&0.009399425108&18&0.008540450681\\
19&0.000847985940&20&0.003771690693&21&0.000873880329\\
27&-0.000419180236&29&-0.000311954497&30&-0.000455650279\\
31&-0.000996284704&35&-0.000865568525&37&0.000206113648
\end{array}
\label{eq:rho-5-coefficients}
\end{equation}
Thus, if \(\rho_r(t)=\sum_m c_{r,m}t^{2m+1}\), the actual correlation
applied to inner products is
\begin{equation}
 \rho_r^{\mathrm{phys}}(z)=\sum_m c_{r,m}z|z|^{2m},
 \qquad r\in\{1,2,3,4,5\}.
 \label{eq:five-physical-preprocessors}
\end{equation}
Exact summation gives the especially simple Wiener norms
\begin{equation} \label{eq:five-preprocessor-norms}
 \|\rho_1\|_{\A}=\|\rho_4\|_{\A}
 =\frac{999999999996}{10^{12}},
 \|\rho_2\|_{\A}=\|\rho_3\|_{\A}
 =\frac{999999999997}{10^{12}},
 \|\rho_5\|_{\A}=\frac{999999999993}{10^{12}}.
\end{equation}
Consequently all preprocessors $\{\rho_r^{\rm phys}\}_{r=1}^5$ are allowable by
Lemma~\ref{lem:wiener}, and each \(K_{\rho_r}\) is a valid Gaussian
phase-rounding kernel for each $1\leq r\leq 5$.  

With the sixth kernel from \eqref{eq:origin-kernel}, label the corresponding top-level kernels by
\begin{equation}
 K_r\colonequals K_{\rho_r}\quad(1\leq r\leq5),
 \qquad K_6\colonequals K_{\mathrm{org}},
 \qquad K_{6,+}\colonequals K_{F,G},\quad K_{6,-}\colonequals K_{F^\#,G^\#}.
 \label{eq:numbered-branch-kernels}
\end{equation}
Thus \(K_6=(K_{6,+}+K_{6,-})/2\) by \eqref{eq:origin-kernel}.  Define also
\begin{equation}\label{wdefs}
\begin{aligned}
 w_1&\colonequals0.983543086263686583,
 &w_2&\colonequals0.0124174018143118538,\\
 w_3&\colonequals0.0030311909556774751,
 &w_4&\colonequals0.000593632758261406976,\\
 w_5&\colonequals0.0000169358532612748746,
 &w_6&\colonequals0.000397752354801503847.
\end{aligned}
\end{equation}
The unnormalized kernel used in Theorem \ref{thm:main} is then
\begin{equation}
 \widehat Q
\colonequals\sum_{r=1}^6w_rK_r
=\sum_{r=1}^5w_rK_r
 \quad+\frac{w_6}{2}K_{6,+}+\frac{w_6}{2}K_{6,-}.
 \label{eq:sparse-kernel-definition}
\end{equation}

\section{Certifying the Upper Bound}\label{sec:prefix}

\subsection{Bounds for the first 120 coefficients}
Let
\begin{equation}
 \widehat Q(z)
 =\sum_{k\geq0}\widehat q_kz|z|^{2k}
 \label{eq:Q}
\end{equation}
be the sum defined in \eqref{eq:sparse-kernel-definition}.  By Lemmas~\ref{lem:inverse-repair} and~\ref{lem:linear},
\[
 K_G^{\C}\leq
 \frac{\sum_{r=1}^6w_r}
 {\widehat q_0-\sum_{k\geq1}|\widehat q_k|}
 \qquad\text{when the denominator is positive.}
\]
The following estimates bound this denominator from below.

The rational part of
every scalar contribution to
\(\widehat q_0,\ldots,\widehat q_{119}\) is computed exactly from
\eqref{eq:scalar-profile}, and the common factor \(\pi\) is enclosed by
directed ball arithmetic.  The contribution of $K_{\mathrm{org}}$ to the sum \eqref{eq:Q} is estimated by bounding the coefficients \eqref{eq:origin-coefficients}.
Here and below, hats mark coefficients of the raw, unnormalized candidate.
 Combining these bounds with the
weights in \eqref{wdefs} gives
\begin{equation}
 \sum_{k=1}^{119}|\widehat q_k|
 <1.074829404498064\cdot10^{-6},
 \qquad
 \widehat q_0>0.71179640805640433452.
 \label{eq:slope-bound}
\end{equation}

The bounds on \eqref{eq:origin-coefficients} integrate over \(0\leq s\leq25\), using the substitution
\(s=t^2\).  A second table bounds the integrals over \(25\leq s\leq36\); the
verifier adds the corresponding interval enclosures before forming products.  If
\(\widetilde a_n,\widetilde b_n\) are the two coefficient sequences whose integrals are truncated
at \(s=36\), and
\(e_n^{(a)}=a_n-\widetilde a_n\),
\(e_n^{(b)}=b_n-\widetilde b_n\) are the omitted sequences, Bessel's
inequality together with $|F|=|G|=1$ give
\begin{equation}
 \|(e_n^{(a)})_{n\geq0}\|_{\ell_2}
 \leq\|F 1_{|z|^2>36}\|_{L_2}
 \leq e^{-18},
 \qquad
 \|(e_n^{(b)})_{n\geq0}\|_{\ell_2}
 \leq\|G 1_{|z|^2>36}\|_{L_2}
 \leq e^{-18}.
 \label{eq:omitted-L2}
\end{equation}
Thus, for any finite index block \(\mathcal I\),
\begin{equation}
 \sum_{n\in \mathcal I}|a_n\overline{b_n}
     -\widetilde a_n\overline{\widetilde b_n}|
 \leq e^{-18}\bigl(\|\widetilde a_{\mathcal I}\|_2
     +\|\widetilde b_{\mathcal I}\|_2\bigr)+e^{-36}.
 \label{eq:joint-omission}
\end{equation}

The file
\nolinkurl{complex_upper_origin_layer_coefficients_n120_analytic.npz} is the accepted interval table.  The companion table is \nolinkurl{origin_tail_25_36.npz}.  Writing \(P=R+iI\), where \(R\) and \(I\) are the real-coefficient
polynomials \(\Re P\) and \(\Im P\), the table generator evaluates them
separately and uses the analytic continuation
\((R+iI)/\sqrt{R^2+I^2}\) on complex integration balls.  This point is
important: projecting the real and imaginary parts of a single complex
ball inside an analytic integration callback would not be a valid
holomorphic continuation.


\subsection{Tail bound for the sixth kernel}
\label{sec:tail}
We now bound $\sum_{n\geq120}|k_n|$, where $k_n$ is defined in \eqref{kndef}.  
The Laguerre integral
\[
 \int_0^\infty e^{-s}\frac{\sqrt{s}L_n^{(1)}(s)}{\sqrt{n+1}}\,ds
 \stackrel{\eqref{eq:origin-coefficients}\wedge\eqref{kzdef}\wedge\eqref{eq:haagerup-phase}}{=}\sqrt{h_n},\qquad\forall\,n\geq0
\]
shows that the coefficient of the constant radial amplitude
\(\ph(z)\) is \(\xi_n\colonequals\sqrt{h_n}\), consistently with
\eqref{eq:origin-coefficients}.  Write
\begin{equation}
\begin{aligned}
 &a_n=f(0)\xi_n+r_n,\qquad b_n=g(0)\xi_n+t_n,\\
 &R_f(z)\colonequals\ph(z)[f(|z|^2)-f(0)]
 \equalscolon \sum_{n\geq0} r_n\phi_n,\\
 &R_g(z)\colonequals\ph(z)[g(|z|^2)-g(0)]
 \equalscolon \sum_{n\geq0} t_n\phi_n.
 \end{aligned}
 \label{eq:endpoint-subtraction}
\end{equation}
Recall from \eqref{bkid} that
\(c_n=a_n\overline{b_n}\) and $k_n=\Re(c_n)$ from \eqref{kndef}.  Substituting
\eqref{eq:endpoint-subtraction} and using
\eqref{eq:endpoint-cancellation}, together with
\(\xi_n=\sqrt{h_n}\), gives
\[
k_n
=\Re\!\left[
 f(0)\overline{g(0)}\,h_n
 +f(0)\xi_n\overline{t_n}
 +r_n\overline{g(0)}\,\xi_n
 +r_n\overline{t_n}
 \right]
=\Re\!\left[
 f(0)\xi_n\overline{t_n}
 +r_n\overline{g(0)}\,\xi_n
 +r_n\overline{t_n}
 \right].
\]
Since \(|f(0)|=|g(0)|=1\) and
\(|\xi_n|=\sqrt{h_n}\), it follows that
\begin{equation} \label{eq:origin-pointwise-tail}
 |k_n|
 \leq
 \sqrt{h_n}\bigl(|r_n|+|t_n|\bigr)
 +|r_n||t_n|.
\end{equation}

Denote $z\in\C$ by $z=x+y\sqrt{-1}$, $x,y\in\R$ and define $\mathcal N\colonequals -\frac{1}{2}(\partial_x^2 + \partial_y^2)+x\partial_x +y\partial_y$.  Write $z=re^{i\theta}$ and $s=r^{2}$.  In polar coordinates, $\Delta=\partial_r^2+r^{-1}\partial_r + r^{-2}\partial^2_{\theta}$.  If $\Phi(z)=e^{i\theta}v(s)$, then $\Delta\Phi=e^{i\theta}(4sv''(s)+4v'(s)-v(s)/s)$ and $(x\partial_x+y\partial_y)\Phi=r\partial_r\Phi=e^{i\theta}2sv'(s)$.  Consequently,
\begin{equation}\label{nformula}
 \mathcal N\bigl(v(s)\ph(z)\bigr)
 =\ph(z)\left[-2sv''(s)+(2s-2)v'(s)+\frac{v(s)}{2s}\right].
\end{equation}
When $v(s)=\sqrt{s}L_n^{(1)}(s)$, the Laguerre differential equation $sL_n^{(1)}{}''(s)+(2-s)L_n^{(1)}{}'(s)+nL_n^{(1)}(s)=0$ with \eqref{nformula} says
\[
-2sv''(s)+(2s-2)v'(s)+\frac{v(s)}{2s}
=\sqrt{s}\Big(L_n^{(1)}(s)+2(s-2)L_n^{(1)}{}'(s)-2sL_n^{(1)}{}''(s)\Big)
=(2n+1)v(s).
\]
%
%

Consequently, by \eqref{phidef},
\begin{equation}\label{neigeq}
\mathcal N \phi_n=(2n+1)\phi_n,\qquad\forall\,n\geq0.
\end{equation}

We apply \eqref{nformula} to \(v=f-f(0)\) and \(v=g-g(0)\), where $f,g$ are defined in \eqref{eq:origin-fg}.  The nonvanishing
estimate \eqref{eq:P-nonzero} and $\Im A,\Im B\neq0$ by \eqref{eq:origin-parameters}, \eqref{eq:origin-parameters-2} imply that \(f\) and \(g\) are twice continuously differentiable on
\([0,\infty)\), so
\[
 f(s)-f(0)=O(s),\qquad g(s)-g(0)=O(s)\qquad(s\downarrow0).
\]
For our choice of $v$, the square brackets in \eqref{nformula}
are then bounded near \(s=0\) and square integrable against \(e^{-s}\,ds\).
Since the phase function is discontinuous at zero, we can use a smooth cutoff near zero (and infinity, since $v$ and its derivatives have polynomial growth) which then converges in the norm $\sqrt{\|\cdot\|_2^2+\|\mathcal N(\cdot)\|_2^2}$ by \eqref{nformula}.
Hence $R_f,R_g\in\operatorname{Dom}(\mathcal N)$.  Then \eqref{neigeq} implies that
\begin{equation}
 \|\mathcal N R_f\|_2^2
 \stackrel{\eqref{eq:endpoint-subtraction}\wedge\eqref{neigeq}}{=}\sum_{n\geq0}(2n+1)^2|r_n|^2,
 \qquad
 \|\mathcal N R_g\|_2^2
 =\sum_{n\geq0}(2n+1)^2|t_n|^2.
 \label{eq:energies}
\end{equation}
Directed integration, including
explicit near-zero and far-tail estimates, proves
\begin{equation}
 \|\mathcal N R_f\|_2^2<12.671,
 \qquad \|\mathcal N R_g\|_2^2<12.674.
 \label{eq:total-energies}
\end{equation}

\begin{equation}
 \sum_{n<120}(2n+1)^2|r_n|^2>11.497,\qquad
 \sum_{n<120}(2n+1)^2|t_n|^2>11.498.
 \label{eq:visible-energies}
\end{equation}
Thus the two tail energies are less than \(1.174\) and \(1.176\).
Lemma \ref{lem:wallis} and \eqref{eq:h-coefficients} imply
\begin{equation}
 W_{120}\colonequals\sum_{n\geq120}\frac{h_n}{(2n+1)^2}
 \leq\frac{1}{48\cdot119^3}.
 \label{eq:weighted-Wallis}
\end{equation}
Applying Cauchy--Schwarz to \eqref{eq:origin-pointwise-tail} while multiplying and dividing by $(2n+1)$ gives
\begin{equation}
 T_6\colonequals\sum_{n\geq120}|k_n|
 \leq\sqrt{1.174\,W_{120}}+\sqrt{1.176\,W_{120}}
 +\frac{\sqrt{1.174\cdot1.176}}{241^2}
 <\frac{523}{2{,}000{,}000}=2.615\cdot10^{-4}.
\label{eq:branch-6-tail}
\end{equation}

\begin{lemma}[Wallis tails]\label{lem:wallis}
Recall $h_k$ from \eqref{eq:h-coefficients}.  For every integer \(N\geq1\),
\begin{equation}
 \sum_{k\geq N}h_k\leq\frac1{4N}.
 \label{eq:Wallis-tail}
\end{equation}
Moreover, for all \(N\geq2\),
\begin{equation}
 \sum_{k\geq N}\frac{h_k}{(2k+1)^2}
 \leq\frac1{48(N-1)^3}.
 \label{eq:weighted-Wallis-general}
\end{equation}
\end{lemma}

\begin{proof}
Wallis' central-binomial estimate
\(\binom{2k}{k}^2/16^k\leq1/(\pi k)\), together with
\eqref{eq:h-coefficients}, gives
\[
 h_k\leq\frac1{4k(k+1)}\qquad\forall\,k\geq1.
\]
Using $1/(k(k+1)) = 1/k - 1/(k+1)$,  $\sum_{k\geq N}1/(k(k+1))=1/N$, proving \eqref{eq:Wallis-tail}.  Also,
\[
 \frac{h_k}{(2k+1)^2}
 \leq\frac1{16k^3(k+1)}
 \leq\frac1{16k^4}.
\]
Comparison with \(\int_{N-1}^{\infty}x^{-4}\,dx\) proves
\eqref{eq:weighted-Wallis-general}.
\end{proof}

\subsection{Tail bound for the fifth (high-degree) preprocessor}
\label{subsec:high-degree-tail}

As in \eqref{three2}, write
\begin{equation}\label{k5def}
K_5(z)
\stackrel{\eqref{krhodef}}{=}\mathfrak{h}(\rho_5^{\rm phys}(z))
\stackrel{\eqref{eq:numbered-branch-kernels}}{\equalscolon}
\sum_{n\geq0}k_{5,n}z|z|^{2n}.
\end{equation}
In this section we will show that
\begin{equation}\label{eq:high-degree-total-tail}
\tau_5\colonequals w_5\sum_{n\geq120}|k_{5,n}|
<6.319618149454563\cdot10^{-8}.
\end{equation}

Write \(\rho_5^{\mathrm{phys}}(z)=zp_5(|z|^2)\), as specified in \eqref{eq:rho-5} and \eqref{eq:rho-5-coefficients} whose largest nonzero coefficient index is $37$.  The polynomial \(x^jp_5(x)^{2j+1}\) in
\eqref{eq:scalar-profile} has degree at most
\begin{equation}
 j+(2j+1)37=75j+37.
 \label{eq:high-degree-block-support}
\end{equation}
Thus the blocks \(0\leq j<67\) have a tail contribution bounded by indices $120\leq n\leq5000$ in \eqref{eq:high-degree-total-tail}.  From \eqref{eq:analytic-h}, we have
$h(\rho_5(t))=\sum_{j\geq0}h_j (\rho_5(t))^{2j+1}.$  Combined with \eqref{k5def}, we have
\[
k_{5,n}
=[t^{2n+1}]\sum_{j\geq0}h_j\rho_5(t)^{2j+1}.
\]
Directed
evaluation gives
\begin{equation}
 \sum_{n=120}^{5000}
 \left|w_5[t^{2n+1}]
       \sum_{j=0}^{66}h_j\rho_5(t)^{2j+1}\right|
 <2.699176355797438\cdot10^{-12}.
 \label{eq:high-degree-finite-tail}
\end{equation}

Recall \(\|\rho_5\|_{\A}\leq1\), so \eqref{eq:five-preprocessor-norms}, \eqref{eq:wiener-submultiplicative} and Lemma~\ref{lem:wallis} bound tail terms omitted from \eqref{eq:high-degree-finite-tail} by
\begin{equation}
w_5\sum_{j\geq67}|h_j|\|\rho_5^{2j+1}\|_{\A}
\leq
w_5\sum_{j\geq67}|h_j|
\leq
 \frac{|w_5|}{4\cdot67}
 <6.3193482318189831\cdot10^{-8}.
 \label{eq:high-degree-outer-tail}
\end{equation}

Combining \eqref{eq:high-degree-finite-tail} and \eqref{eq:high-degree-outer-tail} gives \eqref{eq:high-degree-total-tail}.

\subsection{Tail bound for the first four preprocessors}

In this section, we bound the first four terms in \eqref{eq:sparse-kernel-definition}.  For any \(1\leq r\leq4\), write
\(\rho_r^{\mathrm{phys}}(z)=zp_r(|z|^2)\) and
\(\rho_r^{\mathrm{an}}(\zeta)\colonequals\rho_r(\zeta)=\zeta p_r(\zeta^2)\).  Use the radii
\begin{equation}
 R_1=\frac65,
 \qquad
 R_2=R_3=R_4=\frac{111}{100}.
 \label{eq:scalar-tail-radii}
\end{equation}

By \eqref{hint}, $h(w)$ is equal to the following expression for all $w\in(-1,1)$
\begin{equation}\label{heq}
h_{\Omega}(w)\colonequals w\int_{0}^{\pi/2}\frac{\cos^2\theta}{\sqrt{1-w^2\sin^2\theta}}d\theta.
\end{equation}
Denote $\Omega\colonequals\C\setminus((-\infty,-1]\cup[1,\infty))$.  Then $h_{\Omega}$ is well-defined and holomorphic for any $w\in\Omega$, since for all $0\leq v\leq1$, $1-vw^2\notin(-\infty,0]$ (if $1-vw^2\in(-\infty,0]$, then $vw^2\in[1,\infty)$, violating $w\in\Omega$).  In defining $h_{\Omega}$, we set $\sqrt{1}\colonequals1$.  Meanwhile, we claim that
\begin{equation}\label{rhoclaim}
\rho_r^{\rm an}(\overline{R_r\D})\subset\Omega.
\end{equation}
Given this claim, we can input $w\colonequals \rho_r^{\rm an}(\zeta)$ into $h_{\Omega}$ in \eqref{heq} for any $\zeta\in\overline{R_r\D}$.  To prove \eqref{rhoclaim}, the code verifies that, for all $1\leq r\leq 4$, for any $|\zeta|=R_r$ with $w\colonequals\rho_r^{\rm an}(\zeta)$,
\begin{equation}\label{wineq}
|w|=|\rho_r^{\rm an}(\zeta)|
\stackrel{\eqref{eq:five-physical-preprocessors}}{\leq}\sum_{m}|c_{r,m}|R_r^{2m+1}
<5/4,
\quad\text{and also}\quad
\min_{0\leq v\leq 1}|1-vw^2|>1/100.
\end{equation}
For any $0\leq v\leq 1$, define $F_{r,v}(\zeta)\colonequals1 - v\cdot(\rho_r^{\rm an}(\zeta))^2$.  By \eqref{wineq}, we may define
\[
N_r(v)\colonequals\frac{1}{2\pi\sqrt{-1}}\int_{|\zeta|=R_r}\frac{F_{r,v}'(\zeta)}{F_{r,v}(\zeta)}\,d\zeta,
\]
which is the number of zeros of $F_{r,v}$ in the disc $|\zeta|\leq R_r$ (no zeros occur on the boundary $|\zeta|=R_r$ by \eqref{wineq}).  From \eqref{wineq}, $N_r$ is continuous in $v$ and integer-valued, but $F_{r,0}=1$, so $N_{r}(v)=0$ for all $0\leq v\leq 1$, hence $F_{r,v}(\zeta)$ has no zeros in $|\zeta|\leq R_r$ for all $0\leq v\leq 1$.  So, if $\rho_r^{\rm an}(\zeta)=x\in\R$ with $|x|\geq1$, then $v\colonequals x^{-2}\in(0,1]$ implies $F_{r,v}(\zeta)=1-x^2 x^{-2}=0$, a contradiction.  That is, we proved \eqref{rhoclaim}.  Thus \eqref{heq} holds for $w=\rho_r^{\rm an}(\zeta)$ and $|\zeta|\leq R_r$, so

\[
|h_{\Omega}(w)|
\leq
|w|\int_0^{\pi/2}\frac{\cos^2\theta}{\sqrt{|1-w^2\sin^2\theta|}}d\theta
<\frac{5}{4}\cdot10\int_{0}^{\pi/2}\cos^2\theta d\theta
=25\pi/8<10.
\]
Recall $K_r(z)\stackrel{\eqref{krhodef}}{=}\mathfrak{h}(\rho_r^{\rm phys}(z))\equalscolon\sum_{n\geq0}k_{r,n}z|z|^{2n}$.  Define $H_{r}(\zeta)\colonequals h_{\Omega}(\rho_r^{\rm an}(\zeta))$.  Then $H_r(\zeta)=\sum_{n\geq0}k_{r,n}\zeta^{2n+1}$, and we showed that $|H_r(\zeta)|<10$ for all $|\zeta|=R_r$.

Cauchy's estimate then implies that, for all $1\leq r\leq 4$,
\begin{equation}
 T_r
 \colonequals
 \sum_{n\geq120}|k_{r,n}|
 \leq10\sum_{n\geq120}R_r^{-(2n+1)}
 =\frac{10R_r^{-241}}{1-R_r^{-2}}.
 \label{eq:scalar-tail-reserves}
\end{equation}
Their total raw weighted contribution is
\begin{equation}
 \sum_{r=1}^4|w_r|T_r
 <1.017179644647\cdot10^{-11}.
 \label{eq:low-degree-tail}
\end{equation}

Finally, \eqref{eq:branch-6-tail} and the origin weight give
\begin{equation}
 |w_6|T_6
 <1.040122407805933\cdot10^{-7}.
 \label{eq:weighted-origin-tail}
\end{equation}

\subsection{Putting together the upper bound}\label{sec:assembly}

\begin{proof}[Proof of Theorem \ref{thm:main}, Upper Bound]
The sum of the unnormalized weights is
\begin{equation}
 \mathcal M\colonequals\sum_{r=1}^6 w_r
 \stackrel{\eqref{wdefs}}{=}\frac{1250000000000000121997}
        {1250000000000000000000}
 =1.0000000000000000975976.
 \label{eq:branch-mass}
\end{equation}
The
normalized selector probabilities are \(w_r/\mathcal M\).  Exact comparison gives
\begin{equation}
 \mathcal M<1.0000000000000001.
 \label{eq:mass-upper}
\end{equation}
The combined tail bound for \(K_1,\ldots,K_4,K_6\) is
\begin{equation}
 T_{\mathrm{fix}}
 \colonequals
 \sum_{r=1}^4|w_r|T_r
 +|w_6|\frac{523}{2{,}000{,}000}
 \stackrel{\eqref{eq:low-degree-tail}\wedge\eqref{eq:weighted-origin-tail}}{<}1.040224125770398\cdot10^{-7}.
 \label{eq:fixed-tail}
\end{equation}
Let \(\widehat E=(0,\widehat q_1,\widehat q_2,\ldots)\) be the nonlinear
coefficient vector of \(\widehat Q\).  The prefix, \(\tau_5\), and
fixed-tail estimates give
\begin{equation}
 \|\widehat E\|_1
 \stackrel{\eqref{eq:slope-bound}\wedge\eqref{eq:high-degree-total-tail}\wedge\eqref{eq:fixed-tail}}<1.242047998569650\cdot10^{-6}.
 \label{eq:E-bound}
\end{equation}

Since \(\widehat Q/\mathcal M\) is a convex combination of valid kernels by \eqref{eq:sparse-kernel-definition}, \eqref{eq:five-preprocessor-norms} and \eqref{eq:origin-kernel}, Lemma~\ref{lem:wiener} shows that it is a valid
kernel.  Its analytic representative has the form
\begin{equation}
 q(t)=\frac{\widehat q_0}{\mathcal M}t+r(t),
 \qquad
 \|r\|_{\A}=\frac{\|\widehat E\|_1}{\mathcal M}.
 \label{eq:normalized-candidate}
\end{equation}
Since \(\|\widehat E\|_1<\widehat q_0\) by \eqref{eq:slope-bound} and \eqref{eq:E-bound},
Lemma~\ref{lem:inverse-repair} applies.  It supplies an allowable common
preprocessor $\eta^{\rm phys}$ and an exactly linear valid kernel \(K(z)=\Gamma z\), where
one may take
\begin{equation}
 \Gamma
 \colonequals\frac{\widehat q_0-\|\widehat E\|_1}{\mathcal M}.
 \label{eq:improved-slope}
\end{equation}
Then the deliberately rounded bounds
\eqref{eq:slope-bound}, \eqref{eq:mass-upper}, and \eqref{eq:E-bound} give
\begin{equation}
 \Gamma>
 \frac{0.7117964080564-1.24204800\cdot10^{-6}}
      {1.0000000000000001}
 >0.7117951660083
 >\frac1{1.404898555}.
 \label{eq:coarse-sparse-slope}
\end{equation}

Then displayed bounds
\eqref{eq:slope-bound} and \eqref{eq:E-bound} together with \eqref{eq:branch-mass} give the sharper bound
\begin{equation}
 \Gamma>0.71179516600840569540.
 \label{eq:directed-final-slope}
\end{equation}

Lemma~\ref{lem:linear} now gives
$
 K_G^{\C}\leq\Gamma^{-1}
 <1.404898555.
$
In fact the code verifies that
\begin{equation}
 \Gamma^{-1}<1.404898554745441.
 \label{eq:directed-reciprocal}
\end{equation}
This proves the upper bound of Theorem~\ref{thm:main} since \cite{Haagerup} showed $K_G^\C\leq\gamma_H^{-1}$, $\gamma_H<0.71178980668499814812$, so $\gamma_H^{-1}>1.404909132735795$, and $\gamma_H^{-1}-\Gamma^{-1}>1.057799\cdot 10^{-5}$.
\end{proof}

\section{The Lower Bound}

In this section, function norms and expectations use standard complex Gaussian measure, with density \(\pi^{-n}e^{-\|z\|^2}\) on \(\C^n\).
We write \(\gamma_1\) for this measure when \(n=1\).

Recall that the Hermite space $\Hcal_{p,q}$ consists of the Wick-ordered polynomials
of holomorphic degree $p$ and antiholomorphic degree $q$; these spaces
are mutually orthogonal. Write $P_{p,q}$ for their orthogonal projections
and set $P=P_{1,0}$, $Q=P_{2,1}$. Throughout
\[
 \epsilon=\frac{37}{50},\qquad \rho=\frac{49}{50},\qquad T=T_n=P-\epsilon Q.
\]
On vector-valued functions, \(T_n\) acts on each coordinate.
The lower bound of Theorem \ref{thm:main} follows from the following.
\begin{theorem}\label{thm:main2}
\[
 \sup_{n\geq1}\norm{T}_{L_{\infty}(\C^n)\to L_{1}(\C^n)}<\frac{2857}{4000}<\frac57.
\]
\end{theorem}

\begin{proof}[Proof of lower bound of Theorem \ref{thm:main} assuming Theorem \ref{thm:main2}]
Here \(Z\) is a standard complex Gaussian vector in \(\C^n\).  The first inequality below follows from Grothendieck's inequality (see e.g. \cite{Heilmanb,jones26}).
\begin{equation}\label{klb}
\KG\geq\sup_{n\geq1}
\,\frac{\sup_{g\colon\C^{n}\to S(\C^n)}\E\|Tg(Z)\|_{\ell_{2}(\C^{n})}}{\sup_{f\colon\C^{n}\to S(\C)}\E|Tf(Z)|}
\geq\frac{1}{\sup_{n\geq1}\norm{T}_{L_{\infty}(\C^n)\to L_{1}(\C^n)}}.
\end{equation}
For the second inequality, take \(g(z)=z/\|z\|\) for \(z\neq0\), with any unit value at zero.  Then
\[
 \left\|g-\frac{z}{\sqrt n}\right\|_2^2
 =\E\left(1-\frac{\|Z\|}{\sqrt n}\right)^2
 \leq\E\left(1-\frac{\|Z\|^2}{n}\right)^2
 =\frac1n.
\]
Since \(T_n(z/\sqrt n)=z/\sqrt n\) and \(\|T_n\|_{2\to2}=1\),
\(\|T_ng-g\|_2\leq2/\sqrt n\).  Thus \(\E\|T_ng(Z)\|\to1\) as \(n\to\infty\), proving the second inequality.
Consequently $K_G^{\C}>4000/2857>7/5$ follows by Theorem \ref{thm:main2}.
\end{proof}

We now proceed to prove Theorem \ref{thm:main2}.  Let $f,g\colon\C^n\to S(\C)$.  We will upper bound $\Re\langle f,T_ng\rangle$.  Since $T_n$ is selfadjoint with charge-one range, successive replacements
$f=\operatorname{phase}(T_ng)$ and $g=\operatorname{phase}(T_nf)$
do not decrease $\Rea\langle f,T_ng\rangle$, so we may assume that $f,g$ are charge-one.  (If either polynomial whose
phase is taken is identically zero, then $\Re\langle f,T_ng\rangle\leq0$ and Theorem \ref{thm:main2}
is immediate. Otherwise its zero set has Gaussian measure zero, so its phase is charge-one
a.e.)  Put $h\colonequals(f+g)/2$, $k\colonequals(f-g)/2$.
Denote $t\colonequals\|Ph\|_2$.  Then
\begin{equation}\label{eq:split}
 \begin{gathered}
 |h|^2+|k|^2=1,\qquad \Re(\bar h k)=0,\\
 \Re\langle f,T_ng\rangle
 =t^2-\epsilon\norm{Qh}_2^2+\epsilon\norm{Qk}_2^2-\norm{Pk}_2^2.
 \end{gathered}
\end{equation}
%
Since $|h|\le1$, we have $0\leq t\leq\sup_{\|r\|_2\leq1}|\langle Ph,r\rangle|
=\sup_{\|r\|_2\leq1}|\langle h,Pr\rangle|
\leq \E|W|=\sqrt{\pi}/2<9/10$,
where $W\in\C$ is a standard complex Gaussian.

Section~\ref{sec:moments} proves that, for every $q\in\Hcal_{2,1}$ with $\E|q|^2=1$ and $m\colonequals\E|q|^4$,
\begin{equation}\label{eq:moment-summary}
 2\le m\le62,\qquad
 \E|q|^6\le9m-12+50(m-2)^{3/2},\qquad
 \E|q|\le\rho.
\end{equation}
In particular $\norm{Qu}_2\le\rho$ for any $u\colon\C^n\to \C$ with $|u|\leq1$.
Discarding the negative terms in \eqref{eq:split} gives $\Re\langle f,T_ng\rangle\leq t^2+\epsilon\rho^2=t^2+.710696$.

\section{Certifying the Lower Bound}

\subsection{Reduction to a scalar potential}
Write a point of $\C^n$ as $(z,w)\in\C\times\C^{n-1}$.  (If $n=1$, $w$ can be a single point.)  Recall $t\colonequals\|Ph\|_2$.  By applying a unitary transformation, we may assume $P h(z,w)=tz$.
Denote
\[
 U\colonequals|z|^2,\qquad 
 A\colonequals U-1,\qquad 
 B\colonequals\frac{z^2}{\sqrt2},\qquad
 H\colonequals\frac{z(U-2)}{\sqrt2}.
\]
The functions $1,A,B,z,H,\bar z$ are orthonormal in $L^2(\gamma_1;\C)$.  Recalling $k=(f-g)/2$, set
\[
 c_D(w)\colonequals\mathbb E_z[\overline{D(z)}k(z,w)],\,\,\forall\,w\in\C^{n-1}\qquad
 \xi\colonequals(c_1,c_A,\overline{c_B})^T,\quad 
 \zeta\colonequals(c_z,c_H)^T,\quad 
 d\colonequals c_{\,\overline{z}}.
\]

A \emph{profile} consists of $0<\delta<\epsilon$, real $\ell,\beta$, and
real symmetric positive definite matrices $E\in\mathbb R^{3\times3}$ and
$R\in\mathbb R^{2\times2}$ satisfying
\begin{equation}\label{eq:simple-profile-PSD}
 E\succeq\operatorname{diag}(-1,\epsilon,\epsilon),\quad
 E\succeq\operatorname{diag}(\epsilon-\delta,0,0),\quad
 R\succeq\operatorname{diag}(-1,\epsilon),\quad
 R\succeq\operatorname{diag}(\epsilon,0).
\end{equation}
Write
\[
 \mathcal E(k)\colonequals\xi^*E\xi+\zeta^*R\zeta+\epsilon|c_{\overline{z}}|^2.
\]
Denote $Q_w$ as the operator $Q$ acting only on the $w$ coordinates.  Hermite orthogonality gives
\begin{equation}\label{eq:simple-fiber-domination}
 \epsilon\|Qk\|_2^2-\|Pk\|_2^2
 \le \mathbb E_w\mathcal E(k)+\delta\|Q_wc_1\|_2^2.
\end{equation}
To prove this, denote $\Pi_{p,q}$ as the $w$-variable Hermite projection, and write
$$Pk=\Pi_{1,0}c_1 +z\Pi_{0,0}c_z
,\qquad
Qk=\Pi_{2,1}c_1 +A\Pi_{1,0}c_A
+B\Pi_{0,1}c_B +z\Pi_{1,1}c_z
+H\Pi_{0,0}c_H +\overline{z}\Pi_{2,0}c_{\overline{z}}$$
In the expansion of \(Pk\), the term involving \(z\) is \(z\Pi_{0,0}c_z\), and the remaining terms are \(\Pi_{1,0}c_1\). Denote $\xi_{p,q}\colonequals\Pi_{p,q}\xi$.  Then $\E_w \xi^* E\xi =\sum_{p,q}\E_w\xi_{p,q}^* E \xi_{p,q}$.  Apply the first part of \eqref{eq:simple-profile-PSD} to $\xi_{1,0}$ and the second to $\xi_{2,1}$ (using also $\Pi_{1,0}\overline{c_B} = \overline{\Pi_{0,1}c_B}$) to get
\begin{equation}\label{five1}
\E_w \xi^* E\xi \geq
-\|\Pi_{1,0}c_1\|_{2}^{2}
+\epsilon\|\Pi_{1,0}c_A\|_{2}^{2}
+\epsilon\|\Pi_{0,1}c_B\|_{2}^{2}
+(\epsilon-\delta)\|\Pi_{2,1}c_1\|_{2}^{2}.
\end{equation}
Applying a similar argument to $R$ then gives
\begin{equation}\label{five2}
\E_w \zeta^* R\zeta \geq
-|\Pi_{0,0}c_z|^{2}
+\epsilon|\Pi_{0,0}c_H|^{2}
+\epsilon\|\Pi_{1,1}c_z\|_{2}^{2}.
\end{equation}
Adding \eqref{five1} and \eqref{five2} to $\epsilon\E_w|c_{\overline{z}}|^2\geq\epsilon\|\Pi_{2,0}c_{\overline{z}}\|_2^2$ and adding $\delta\|Q_wc_1\|_2^2$ proves \eqref{eq:simple-fiber-domination}.

Put $L(z)\colonequals z(\ell+\beta(U-2))$. Since $\mathbb E\bar z h=t$ and
$\|z(U-2)\|_2^2=2$, we have
\[
0\leq\epsilon\|Qh+\beta\sqrt{2}H/\epsilon\|_2^2
=\epsilon\|Qh\|_2^2+2\Re \E[\overline{\beta z(U-2)}Qh]
+2\beta^2\|H\|_2^2 / \epsilon.
\]
That is,
\begin{equation}\label{qhineq}
 -\epsilon\|Qh\|_2^2
 \le 2\Re\,\mathbb E[\overline{\beta z(U-2)}h]
       +\frac{2\beta^2}{\epsilon}.
\end{equation}
Define the particular residual function
\begin{equation}\label{qsdeq}
 q\colonequals\delta Q_wc_1,\qquad \sigma\colonequals\|q\|_2\le\delta\rho,\qquad
 \delta\|Q_wc_1\|_2^2
 =2\Re\E_w\bar q c_1-\sigma^2/\delta.
\end{equation}
The bound on $\sigma$ follows from $|c_1|\le1$ and
\eqref{eq:moment-summary}.  The equality follows since $q$ is in the range of $Q_w$ and $c_1=Q_w c_1 + (I-Q_w)c_1$, so $\E_w \bar q c_1=\E_w \bar q Q_w c_1=\delta \|Q_w c_1\|_2^2\in\R$ and $\sigma^2/\delta=\delta\|Q_w c_1\|_2^2$.

For a complex number $\omega\in\C$, define the scalar potential
\begin{equation}\label{eq:simple-potential-primal}
 \Psi(\omega)\colonequals\frac{2\beta^2}{\epsilon}
 +\sup_{f_0,g_0\colon\C\to S(\C)}
 \Big(\mathcal E(k_0)+2\Re\mathbb E_z[\bar L h_0]+2\Re\E_z [\bar\omega k_0(z)]\Big).
\end{equation}
Here $f_0,g_0$ depend only on $z$, with $h_0\colonequals(f_0+g_0)/2$, $k_0\colonequals(f_0-g_0)/2$.
Bessel's inequality makes the supremum finite.
We claim that $\Psi(\omega)=\Psi(|\omega|)$, which follows since the map
\[
 (f_0(z),g_0(z))\mapsto
 e^{-i\theta}(f_0(e^{i\theta}z),g_0(e^{i\theta}z)),
\]
multiplies $\xi$ by $e^{-i\theta}$, $\zeta$ is fixed, and $d$ is multiplied by
$e^{-2i\theta}$. Hence the $\mathcal E$ and $L$ terms in \eqref{eq:simple-potential-primal} stay fixed while the coefficient of $\overline\omega$ can be rotated by any phase. Substituting \eqref{qhineq} and \eqref{qsdeq} into \eqref{eq:split}, then using
\eqref{eq:simple-fiber-domination} and this observation with $\omega\colonequals q(w)$,
\begin{equation}\label{eq:simple-moment-target}
 \Re\langle f,T_ng\rangle
 \le t^2-2\ell t+\mathbb E_w\Psi(|q|)-\frac{\sigma^2}{\delta}.
\end{equation}

The potential $s\mapsto\Psi(s)$, $s\in\R$, is convex and $2$-Lipschitz: it is a supremum of affine
functions with slopes $2\Re c_1\in[-2,2]$.
It is even by the claim, hence nondecreasing on
$[0,\infty)$.

\subsection{Finite optimization}\label{sec:certificate}
We optimize \eqref{eq:simple-moment-target}.  Set $K\colonequals(\Rea\xi,\operatorname{Im}\xi,\Rea\zeta,\operatorname{Im}\zeta,
\Rea d,\operatorname{Im}d)^T\in\R^{12}$.
Then $\mathcal E(k_0)=K^TMK$ and
$K_j=\Re\E_z\overline{\mathcal B_j}k_0$ for all $1\leq j\leq 12$ where
\[
 M\colonequals\diag(E,E,R,R,\epsilon I_2)\in\R^{12\times12},\qquad
 \mathcal B\colonequals(1,A,B,\ i,iA,-iB,\ z,H,\ iz,iH,\ \bar z,i\bar z).
\]
Let $v\in\R^{12}$.  Note that $2v^T K - v^T M^{-1}v=K^T MK-(v-MK)^T M^{-1}(v-MK)\leq K^{T}MK$ with equality when $v=MK$, so $K^TMK=\sup_{v\in\R^{12}}(2v^TK-v^TM^{-1}v)$.  Adding $2v^T K$ to the last two terms in \eqref{eq:simple-potential-primal} yields $2\Re\E_z[\overline{\omega+\mathcal Bv}k_0 + \bar Lh_0]=\Re\E_z[\overline{\omega+\mathcal B v+L}f_0+\overline{-\omega-\mathcal B v+L}g_0]$. 
For any $w>0$, we have $0\leq(|a|-w)^2=|a|^2 - 2|a|w+w^2$, so that $|a|\le(w+|a|^2/w)/2$.  Applying this for positive step functions $w_+,w_-$, gives in \eqref{eq:simple-potential-primal}, for any $s\in\R$
\begin{equation}\label{eq:price}
 \Psi(s)\le\sup_{v\in\mathbb R^{12}}R_s(v),\qquad
 R_s(v)\colonequals\frac{2\beta^2}{\epsilon}
 +\frac12\E_z\sum_{\eta\in\{-1,1\}}
 \left(w_\eta+\frac{|s+\mathcal Bv+\eta L|^2}{w_\eta}\right)-v^TM^{-1}v.
\end{equation}
Then $R_s(v)$ is a quadratic polynomial in $v$. At each rational node $s_j$, the
code expands it as $c_j+2b_j^Tv-v^TD_jv$ and certifies
$
 D_j\succ0$, $c_j+b_j^TD_j^{-1}b_j\le A_j.
$
Then $\Psi(s_j)\le A_j$ since $R_{s_j}(v)=c_j+b_j^T D_j^{-1}b_j -(v-D_j^{-1}b_j)^{T}D_j(v-D_j^{-1}b_j)\leq c_j+b_j^T D_j^{-1}b_j$.  Convexity of $\Psi$ bounds $\Psi$ below the line segment between consecutive nodes
$0=s_0<\cdots<s_J=S$; the $2$-Lipschitz bound gives
$\Psi(s)\le A_J+2(s-S)$ for $s\ge S$.

The code supplies polynomials
$p(s)=a_0+a_1s+a_2s^2+a_4s^4+a_6s^6$, with $a_1,a_6\ge0$,
lying above these line segments and the entire affine tail. Thus $p\ge\Psi$.
For $\sigma>0$ put $x^2\colonequals\E|q/\sigma|^4-2$. By \eqref{eq:moment-summary},
\begin{equation}\label{eq:moment-box}
 \begin{split}
 \E\Psi(|q|)-\sigma^2/\delta\le J_p(\sigma,x)
 \colonequals{}&
 a_0+\rho a_1\sigma+(a_2-\delta^{-1})\sigma^2\\
 &+a_4(2+x^2)\sigma^4+a_6(6+9x^2+50x^3)\sigma^6.
 \end{split}
\end{equation}
Here $0\le\sigma\le\delta\rho$, $0\le x\le\sqrt{60}<31/4$.
If $\sigma=0$, then $q=0$ and $\E\Psi(|q|)\le p(0)$.
The coefficients $a_2,a_4$ may have either sign because those moments are exact.
On each rectangle of a complete cover of $[0,\delta\rho]\times[0,31/4]$,
one such polynomial is checked to satisfy $J_p\le C$. For fixed $\sigma$,
the $x$-dependence is $ax^2+bx^3$ with $b\ge0$, so its maximum is at an endpoint by the first derivative test.  Consequently \eqref{eq:simple-moment-target} becomes
\begin{equation}\label{eq:parabola}
 \Re\langle f,T_ng\rangle\le t^2-2\ell t+C.
\end{equation}

The table gives the nine certified pairs $(\ell_i,C_i)$ and, at $i=0$,
the elementary pair $(0,\epsilon\rho^2)$. All decimals are exact rationals;
the remaining profile parameters and full certificates are supplied with the verifier code.
\begin{center}\small
\begin{tabular}{rrr@{\qquad}rrr}\toprule
$i$&$\ell_i$&$C_i$&$i$&$\ell_i$&$C_i$\\\midrule
0&0&0.710696&5&0.610085330&1.080649\\
1&0.121848821&0.725165&6&0.697084166&1.197222\\
2&0.232185104&0.760598&7&0.782190305&1.325111\\
3&0.333426881&0.816862&8&0.833308193&1.408127\\
4&0.492185281&0.946281&9&0.884543194&1.495777\\\bottomrule
\end{tabular}
\end{center}
Set
$
 \tau_0\colonequals0$, $
 \tau_i\colonequals\frac{C_i-C_{i-1}}{2(\ell_i-\ell_{i-1})}$, $\forall$ $1\le i\le9$, $\tau_{10}\colonequals9/10.$
The checker verifies that these endpoints increase and that
$F_i(t)=t^2-2\ell_i t+C_i$ is strictly below $2857/4000$ at both
endpoints of $[\tau_i,\tau_{i+1}]$. Convexity proves the same bound on
each interval, completing Theorem~\ref{thm:main2}.

\subsection{The cubic moment lemma}\label{sec:moments}
\begin{proof}[Proof of \eqref{eq:moment-summary}]
We first prove $\kappa_6\le68\kappa_4^{3/2}$ in \eqref{kineq}.  We then take $X=\sqrt2\Re q$ for $q\in\Hcal_{2,1}$ with $\|q\|_2=1$, derive the moment bounds in \eqref{eq:moment-summary}, and conclude with $\E|q|\le49/50$.

A real symmetric tensor $f$ of order $3$ on $\R^n$ is an array $f=(f_{abc})_{1\leq a,b,c\leq n}$ whose entries are invariant under permutation, e.g. $f_{abc}=f_{acb}=f_{bca}=f_{bac}=f_{cba}=f_{cab}$. Let $Y_1,\ldots,Y_n$ be independent real standard Gaussian random variables.  The Wick cubic $I_3(f)$ is $I_3(f)=\sum_{1\leq a,b,c\leq n}f_{abc}(Y_aY_bY_c - 1_{\{a=b\}}Y_c- 1_{\{a=c\}}Y_b- 1_{\{c=b\}}Y_a)$.  Denote $\|f\|^2\colonequals\sum_{1\leq a,b,c\leq n}f_{abc}^2$.  For a real symmetric tensor $f$, denote $X\colonequals I_3(f)$. Define
\[
 F_{ab,c}\colonequals f_{abc},\quad 
 G\colonequals F^TF,\quad 
 S\colonequals \operatorname{sym}_4(FF^T),\quad
 s\colonequals \|G\|^2,\quad 
 \]
 \[
 h\colonequals \|S\|^2,\quad 
 D\colonequals FG,\quad 
 U\colonequals SF,
\]
where $\operatorname{sym}_4$ averages all permutations of the four tensor indices
and norms are Hilbert--Schmidt, and $F$ is treated as an $n^2\times n$ matrix. Also put
\[
 V\colonequals\|D\|^2=\operatorname{tr}(G^3),\qquad B_0\colonequals\|F^T\operatorname{vec}(G)\|^2,\qquad 
 A_0\colonequals\sum_{a,b,c,d,e=1}^nf_{abc}f_{dec}G_{ad}G_{be}.
\]
Write $\kappa_4\colonequals\E X^4-3(\E X^2)^2$ and
$\kappa_6\colonequals\E X^6-15\E X^4\E X^2+30(\E X^2)^3$.  We check
\begin{equation}\label{eq:wick}
 \begin{split}
 \kappa_4&=1296(s+3h/2),\\
 \frac{\kappa_6}{1296^{3/2}}
 &=\frac{135}{2}\operatorname{tr}(S^3)+25V+\frac{45}{2}B_0
   +\frac{165}{2}\langle D,U\rangle+45\|U\|^2+45(A_0-V).
 \end{split}
\end{equation}
Wick's rule and cumulant cancellation express $\kappa_r$ as the sum over
connected loopless cubic multigraphs on $r$ labelled vertices.
An edge is a summed tensor index; a graph with multiplicities $e_{ij}$
has weight $(3!)^r/\prod_{i<j}e_{ij}!$.
The program enumerates all such graphs for $r=4,6$, expands the right-hand
sides of \eqref{eq:wick} into the same contractions, and compares exact
rational coefficients after relabelling vertices. Symmetry of $f$ makes
this a formal identity, with no dimension bound or numerical tensor tests.

Since $\mathrm{sym}_4$ averages over norm-preserving index permutations, $h=\|\mathrm{sym}_4(FF^T)\|^2\leq\|FF^T\|^2=s$, so $0\le h\le s$.  If $G$ has eigenvalues $(\lambda_i)$, then $s=\sum\lambda_i^2$, so $\|F\|_{\mathrm{op}}^2=\|G\|_{\mathrm{op}}=\max\lambda_i\leq\sqrt{s}$.
Cauchy--Schwarz and the eigenvalue bounds give
\[
 A_0\le V\le s^{3/2},\quad B_0\le s^{3/2},\quad
 \operatorname{tr}(S^3)\le h^{3/2},\quad
 \|U\|^2\le h\sqrt s,\quad \langle D,U\rangle\le s\sqrt h.
\]
For the first inequality, diagonalize $G$: then
$A_0=\sum_{ijk}\lambda_i\lambda_j f_{ijk}^2\le\sum_{ijk}\lambda_k^2f_{ijk}^2=V$
by $2\lambda_i\lambda_j\le\lambda_i^2+\lambda_j^2$ and symmetry.  Then $V=\sum \lambda_i^3\leq\max\lambda_i \cdot s\leq s^{3/2}$.  Next, $B_0\leq\|F\|_{\mathrm{op}}^2\|G\|^2\leq s\cdot\sqrt{s}$.  The $\mathrm{tr}(S^3)$ bound follows similarly by diagonalizing $S$.  The $\|U\|^2$ bound mimics the $B_0$ bound.  The final bound follows by Cauchy--Schwarz.  Thus, with $y\colonequals\sqrt{h/s}\in[0,1]$,
\begin{equation}\label{kineq}
 \frac{\kappa_6}{\kappa_4^{3/2}}
 = \frac{\kappa_6/s^{3/2}}{\kappa_4^{3/2}/s^{3/2}}
 \stackrel{\eqref{eq:wick}}{\le}\frac{N(y)}{(1+3y^2/2)^{3/2}}<68,\qquad
 N(y)\colonequals
 \frac{135}{2}y^3
 +\frac{95}{2}+\frac{165}{2}y+45y^2.
\end{equation}
The last step is a rational polynomial check:
$68^2(1+3y^2/2)^3-N(y)^2>0$ on $[0,1]$.
If $s=0$, then $f=0$ and the cumulant inequality is immediate.

For any $q\in\Hcal_{2,1}$ with $\|q\|_2=1$, circular symmetry makes
$X=\sqrt2\Re q=(q+\bar q)/\sqrt{2}$ a normalized real third-chaos variable, with 
$\E X^2=1$, $\E X^4=(3/2)m$, $\E X^6 = (5/2)\E|q|^6$, so
\begin{equation}\label{k46}
 \kappa_4=\tfrac32(m-2),\qquad
 \kappa_6=\tfrac52(\E|q|^6-9m+12),\qquad m\colonequals\E|q|^4.
\end{equation}
Since $\kappa_4\geq0$ by \eqref{eq:wick}, it follows that $m\ge2$.  Also
$\E|q|^6\le9m-12+50(m-2)^{3/2}$ by \eqref{kineq} and \eqref{k46}, using $68(3\sqrt6/10)<50$.
Also $h\le s\le(\operatorname{tr}G)^2=\|f\|^4=1/36$ since $1=\E X^2=3!\|f\|^2$.  So
$\kappa_4\le90$ by \eqref{eq:wick} and therefore $m\le62$ by \eqref{k46}.

Finally, put $Y\colonequals|q|$, $\mu=\E Y$ and $n_6\colonequals\E Y^6$.
Cauchy--Schwarz, the triangle inequality, $\E Y^4(Y+1)^2 = \E Y^6+2\E Y^5 + \E Y^4$, $\E Y^5\leq\sqrt{\E Y^6 \E Y^4}$ and $2\sqrt{n_6m}\le n_6/6+6m$ give
\begin{equation}\label{m1eq}
\begin{aligned}
 (m-1)^2
 &=\big(\E(Y-1)[Y^2(Y+1)]\big)^2
 \leq \E(Y-1)^2 \cdot\E Y^4(Y+1)^2\\
 &\le2(1-\mu)(\sqrt{n_6}+\sqrt m)^2
 \le2(1-\mu)(7n_6/6+7m).
\end{aligned}
\end{equation}
For $x=\sqrt{m-2}$, the $\E|q|^6$ bound of \eqref{eq:moment-summary} gives $n_6\leq 6+9x^2+50x^3$, so
\begin{equation}\label{m1eq2}
\begin{aligned}
 25(m-1)^2-7n_6/6-7m
 &\ge25x^4-\tfrac{175}{3}x^3+\tfrac{65}{2}x^2+4\\
 &=25(x^2 - 7x/6-1/12)^2
 +(95/36)(x-35/38)^2+181/114>0.
 \end{aligned}
\end{equation}
Combining \eqref{m1eq} and \eqref{m1eq2} yields $1-\mu\geq\frac{(m-1)^2}{2((7/6)n_6 + 7m)}\geq1/50$, so $\mu\le49/50$. 
For $q=Qu$ and $|u|\le1$,
$\|q\|_2^2=\Re\E\bar q u\le\E|q|\le(49/50)\|q\|_2$, completing \eqref{eq:moment-summary}.

\end{proof}

\section{Supporting Numerical Codes}

The numerical portion of the upper bound of Theorem \ref{thm:main} can be run with the commmand: \verb!python3 audit_best7_full.py! with the codes available at: \url{https://github.com/sheilman77/grothendieck_cx}.  The numerical portion of the lower bound for Theorem \ref{thm:main} can be run with the command: \verb!python3 check_complex_lower.py!.

\noindent
\textbf{Acknowledgements}.  GPT-6 Astra assisted in the preparation of this manuscript.

S.H. is supported by NSF Grant CCF AF 2448108.  G.M.\ is supported by the European Research Council through an ERC Starting Grant (Grant agreement No.~101077455, ObfusQation) and funded by the Deutsche Forschungsgemeinschaft (DFG, German Research Foundation) under Germany's Excellence Strategy - EXC 2092 CASA – 390781972.

\end{document}